\documentclass[12pt,a4paper,reqno]{amsart}

\usepackage[a4paper,inner=1.26in,outer=1.26in,top=1.4in,bottom=1in]{geometry}
\usepackage{amsmath,amssymb,amsfonts}
\usepackage{graphicx,hyperref}
\usepackage{xcolor}

\numberwithin{equation}{section}

\newtheorem{theorem}{Theorem}[section]
\newtheorem{proposition}[theorem]{Proposition}
\newtheorem{lemma}[theorem]{Lemma}

\theoremstyle{definition}

\theoremstyle{remark}
\newtheorem*{remark}{Remark}

\newcommand{\R}{\mathbb R}

\newcommand{\dd}{\mathrm d}

\title[Examples of Noncompact Quasinilpotent Hankel Operators]
      {Examples of Noncompact Quasinilpotent Hankel Operators}

\author{Yuanqi Sang}
\address{School of Mathematics, Southwestern University of Finance and Economics,
  Chengdu, 611130, China}
\email{sangyq@swufe.edu.cn}

\date{}

\begin{document}

\begin{abstract}
For every nonzero real number $a,$ we prove that the integral Hankel operator on $L^2(\mathbb R_{+})$ with kernel
$k(t)=t^{-1+ia}$ is noncompact quasinilpotent.  
This answers in the affirmative a question by Peller \cite{Peller1998} about the existence of noncompact quasinilpotent Hankel operators.
Moreover, we obtain the corresponding Hankel operators on the Hardy space with 
symbols $\left((1-z)/(1+z)\right)^{ia}.$ 
\end{abstract}

\maketitle

\section{Introduction}

Let $\mathbb T$ denote the unit circle in $\mathbb C$.
The Hardy space $H^{2}(\mathbb T)$
is the closed subspace of $L^2(\mathbb T)$ consisting
of functions whose negative Fourier coefficients vanish.
Let $P_{+}$ be the orthogonal projection from $L^2(\mathbb T)$ onto $H^{2}(\mathbb T).$ 
For $f\in L^{\infty},$ the space of essentially bounded measurable functions on $\mathbb{T}$, 
the Hankel operator
$H_{f}$ is defined by:
\begin{eqnarray*}
H_{f}:& x\in H^{2}\longmapsto P_{+}(f (Jx))\in H^{2},
\end{eqnarray*}
where $J$ is the flip operator on $L^{2}(\mathbb T)$ defined by
$Jf(z)=\bar{z}f(\bar{z}).$
With respect to the standard basis $\{z^{n}\}^{\infty}_{n=0}$ of the Hardy space $H^{2}(\mathbb T)$, 
the Hankel operator $H_f$ has the matrix representation
\begin{align}\label{THm}
\begin{pmatrix}
f_1 & f_2 & f_3 & f_4 & \cdots\\
f_2 & f_3 & f_4 & f_5 & \cdots\\
f_3 & f_4 & f_5 & f_6 & \cdots\\
f_4 & f_5 & f_6 & f_7 & \cdots\\
\vdots & \vdots & \vdots & \vdots & \ddots
\end{pmatrix},
\end{align}
where $f_n$ is the $n$-th Fourier coefficient of the symbol $f$.
In particular,
$H_{f}=0$ if and only if $f\in\overline{H^{2}(\mathbb T)}.$

The integral Hankel operator $\boldsymbol{H}_{k}$ with kernel $h$
on $L^2(\mathbb R_{+})$ is defined by
\begin{align*}
(\boldsymbol{H}_{k}f)(x)=\int_{\mathbb R_{+}}k(x+y)f(y)dy.
\end{align*}
In general, $k$ may be a distribution.
If $\boldsymbol{H}_{k}$ is bounded, then $\boldsymbol{H}_{k}$ is  unitarily equivalent  to  a Hankel operator on $H^2(\mathbb T)$
(see \cite[Theorem 8.9]{Peller2003}).

An operator $T$ on a Hilbert space $H$ is called nilpotent if there exists a positive integer $n$ such that $T^n=0.$
It is called quasinilpotent if its spectrum $\sigma(T)$ equals $\{0\}.$

Power \cite{Power1984} proved that there does not exist a 
nonzero nilpotent Hankel operator
and asked whether there exists a nonzero quasinilpotent Hankel operator.
The question was answered affirmatively by Megretskii \cite{Megretskii1990}, 
who constructed a nonzero compact quasinilpotent
Hankel operator. Peller \cite{Peller1998} asked whether there exist noncompact quasinilpotent Hankel operators.
We give examples of noncompact quasinilpotent Hankel operators in 
Theorem \ref{main1} and Theorem \ref{main2}. We summarize our results as follows.

\noindent\textbf{Main results.}
Let $a$ be a nonzero real number. Define $k(t)=t^{-1+ia}$ and 
$\psi(z)=\left(\frac{1-z}{1+z}\right)^{ia}.$ The operators
$\boldsymbol{H}_k$ and $H_{\psi}$ are noncompact quasinilpotent on $L^2(\mathbb R_{+})$ 
and $H^{2}(\mathbb T),$ respectively.

Recall the unitary Mellin transform $\mathcal M$ (see \cite[p.166]{Nikolski2002}) defined by
\begin{align}\label{mellin}
 \mathcal M:L^2(\mathbb R_+)\longrightarrow L^2(\R),
 \qquad
 (\mathcal M f)(s)=\frac{1}{\sqrt{2\pi}}
 \int_{\mathbb R_{+}} x^{-1/2-is}f(x)\,dx.
\end{align}
Moreover,
\begin{align}\label{inversemellin}
 \mathcal M^{-1}:L^2(\mathbb R)\longrightarrow L^2(\mathbb R_{+}),
 \qquad
 (\mathcal M^{-1} g)(s)=\frac{1}{\sqrt{2\pi}}
 \int_{\mathbb R} s^{-1/2+it}g(t)\,dt.
\end{align}
Define $(\boldsymbol{M}f)(x)=
\frac{\Gamma(1/2-i x)}{|\Gamma(1/2-i x)|}(\mathcal Mf)(-x).$
For an integral Hankel operator $\boldsymbol{H}_{k},$ let $A=\boldsymbol{M}\boldsymbol{H}_{k}\boldsymbol{M}^{-1}.$
Yafaev \cite{Yafaev2019} obtained the following formula:
\begin{align*}
A=v(X)s(D)v(X),
\end{align*}
where $(Xu)(x)=xu(x), D=-i\frac{d}{dx},v(x)=\frac{\sqrt{\pi}}{\sqrt{\cosh(\pi x)}}$ 
and $s$ is the sign-function of $\boldsymbol{H}.$
The formula suggests considering weighted composition operators on $L^{2}(\mathbb R).$

We find a class of noncompact quasinilpotent weighted composition operators on $L^{2}(\mathbb R).$
Under the Mellin transform, these operators are unitarily equivalent to a class of integral Hankel operators.
These integral Hankel operators can be viewed as generalized Carleman operators. Finally, we calculate the symbols 
of the corresponding Hankel operators on the Hardy space $H^{2}(\mathbb T).$

\section{Proof of the Main Results}
For $\phi\in C_{0}(\mathbb R)\cap L^1(\mathbb R)$ and $a\in \mathbb{R}\setminus\{0\}$,
define the operator $W_{\phi,a}: L^2(\mathbb R)\to L^2(\mathbb R)$ 
by
\begin{align}
(W_{\phi,a}f)(x)=\phi(x)f(x-a)\qquad f\in L^2(\mathbb R).
\end{align}
\begin{proposition}\label{w}
If $\phi\not\equiv 0,$ then
the operator $W_{\phi,a}$ is not compact and $\sigma(W_{\phi,a})=\{0\}.$
\end{proposition}
\begin{proof}
Suppose that $\phi\in C_{0}(\mathbb R).$ 
For $\varepsilon \in (0,\|\phi\|_{\infty}),$ the set
\begin{align}
E_{\varepsilon}=\{x\in \mathbb R:|\phi(x)|\geqslant \varepsilon\}
\end{align}
is compact and has positive measure.
We can divide $E_\varepsilon$ into
countably many pairwise disjoint  measurable sets $I_n,$ each of positive measure.
Let $h_n={\chi_{I_n -a}}/{|I_n|^{\frac{1}{2}}}.$ We have $\|h_n\|=1.$
A direct calculation yields $W_{\phi,a}h_n=|I_n|^{-\frac{1}{2}}\chi_{I_n}\phi.$ 
Thus $\|W_{\phi,a}h_n\|\geqslant \varepsilon.$ If $n\neq m,$ then 
\begin{align*}
\langle W_{\phi,a}h_n,W_{\phi,a}h_m \rangle=0.
\end{align*}
It follows that 
\begin{align*}
\|W_{\phi,a}h_n-W_{\phi,a}h_m\|^2=\|W_{\phi,a}h_n\|^2+\|W_{\phi,a}h_m\|^2\geqslant 2\varepsilon^2.
\end{align*}
Therefore, $W_{\phi,a}$ is not compact.

A direct calculation yields
\begin{align*}
(W^{n}_{\phi,a}f)(x)=f(x-na)\prod_{k=0}^{n-1}\phi(x-ka).
\end{align*}
Since translations are unitary operators on $L^2(\mathbb R),$ we have 
\begin{align}\label{norm}
\|W^{n}_{\phi,a}\|=\mathop{\mathrm{ess\,sup}}_{x\in\mathbb R}\prod_{k=0}^{n-1}|\phi(x-ka)|.
\end{align}
Let $N_{\varepsilon}=1+\lfloor \frac{diam (E_\varepsilon)}{|a|}\rfloor.$ For every $x\in \mathbb R,$
at most $N_{\varepsilon}$ of the points $x,x-a,\cdots,x-(n-1)a$ lie in $E_{\varepsilon}.$ For
$n>N_{\varepsilon},$ we have
$\prod_{k=0}^{n-1}|\phi(x-ka)|\leqslant \|\phi\|_{\infty}^{N_\varepsilon}\varepsilon^{n-N_\varepsilon}.$
Therefore,
\begin{align*}
\|W^{n}_{\phi,a}\|^{\frac{1}{n}}\leqslant (\frac{\|\phi\|_{\infty}}{\varepsilon})^{\frac{N_\varepsilon}{n}}\varepsilon\to \varepsilon\qquad \text{as} \ n\to\infty.
\end{align*}
Since $\varepsilon$ is arbitrary in $(0,\|\phi\|_{\infty}),$ $r(W_{\phi,a})=0.$
\end{proof}

For $r>0,$ define
\begin{align*}
(V_{r}f)(x)=\sqrt{r}f(rx),\qquad f\in L^2(\mathbb R_{+}).
\end{align*}
and 
$V^{-1}_{r}=V_{r^{-1}}.$
Define
\begin{align*}
(\boldsymbol{H}f)(x)
= &(\mathcal M^{-1}W_{\phi,a}\mathcal M f)(x),\qquad x>0.
\end{align*}
We denote by $C_c^\infty(\mathbb R_+)$ the space of all infinitely
differentiable functions with compact support in $\mathbb R_+$. Recall that
$C_c^\infty(\mathbb R_+)$ is dense in $L^2(\mathbb R_+)$.
\begin{theorem}\label{main1}
Let $\phi\in C_{0}(\mathbb R)\cap L^1(\mathbb R)$ and $a\in \mathbb R\setminus\{0\}.$ 
If
$\boldsymbol{H}$ is an integral Hankel operator, then
\begin{align}
(\boldsymbol{H}f)(x)=c\int_{\mathbb R_{+}}(x+y)^{-1+ia}f(y)dy \qquad f\in L^{2}(\mathbb R_{+}),
\end{align}
where $c$ is constant.
In this case,
\begin{align*}
\phi(s)=\frac{c\Gamma(1/2-is)\Gamma(1/2+i(s-a))}{\Gamma(1-ia)}\qquad s>0.
\end{align*} 
\end{theorem}
\begin{proof}
Assume that $\boldsymbol{H}$ is an integral Hankel operator 
such that 
\begin{align*}
(\boldsymbol{H}f)(x)=\int_{\mathbb R_{+}}k(x+y)f(y)dy.
\end{align*}
For $h\in C_c^\infty(\mathbb R_+),$ we have
\begin{equation}\label{Vr}
\begin{aligned}
(\mathcal M V_{r}h)(x)&=\frac{1}{\sqrt{2\pi}}\int_{\mathbb R_{+}}s^{-\frac{1}{2}-ix}\sqrt{r}h(rs)ds\\
&=\frac{r^{ix}}{\sqrt{2\pi}}\int_{\mathbb R_{+}}y^{-\frac{1}{2}-ix}h(y)dy\\
&=r^{ix}(\mathcal{M}h)(x).
\end{aligned}
\end{equation}
Using \eqref{Vr}, we have
\begin{equation*}
\begin{aligned}
(\mathcal M \boldsymbol{H}V_{r}h)(x)&=(W_{\phi,a}\mathcal M V_{r}h)(x)\\
&=\phi(x)r^{i(x-a)}(\mathcal Mh)(x-a)\\
&=r^{i(x-a)}(W_{\phi,a}\mathcal M h)(x)
\end{aligned}
\end{equation*}
and 
\begin{align*}
(\mathcal M V_{r}\boldsymbol{H} h)(x)=r^{ix}(\mathcal M \boldsymbol{H}h)(x)
=r^{ix}(W_{\phi,a}\mathcal M h)(x)
\end{align*}
Since $\mathcal M$ is unitary and $C_c^\infty(\mathbb R_+)$ is dense in $L^2(\mathbb R_+),$
\begin{align}\label{VH}
\boldsymbol{H}V_{r}=r^{-ia}V_{r}\boldsymbol{H}.
\end{align}
A direct calculation gives
\begin{equation}\label{Vr-1}
\begin{aligned}
(V_{r^{-1}}\boldsymbol{H}V_{r}h)(x)
&=r^{-\frac{1}{2}}\int_{\mathbb R_{+}}k(r^{-1}x+y)\sqrt{r}h(ry)dy\\
&=\int_{\mathbb R_{+}}k(r^{-1}x+y)h(ry)dy\\
&=\int_{\mathbb R_{+}}r^{-1}k(r^{-1}(x+t))h(t)dt.
\end{aligned}
\end{equation}
Combining \eqref{VH} and \eqref{Vr-1}, we have
$r^{1-ia}k(rx)=k(x).$
Let $r=e^b$ and $x=e^{t}.$ We have
\[e^{(1-ia)(b+t)k(e^{b+t})}=e^{(1-ia)t}k(e^t).\] Let $g(t)=e^{(1-ia)t}k(e^t).$ We have
$g(b+t)=g(t),\forall \ b,t>0.$  Therefore, $g$ is constant.
It follows that 
\[k(t)=ct^{-1+ia},\]
where $c$ is constant. Therefore,
\begin{align*}
(\boldsymbol{H}h)(x)=c\int_{\mathbb R_{+}}(x+y)^{-1+ia}h(y)dy.
\end{align*}
Using the identity
$\int_0^\infty\frac{y^{-1/2}}{x+y}\dd y=\pi x^{-1/2},$ we have
\begin{equation*}
 \int_0^\infty\!\int_0^\infty
 \frac{x^{-1/2}|h(y)|}{x+y}\dd y\dd x
 =\pi\int_0^\infty y^{-1/2}|h(y)|\dd y<\infty
\end{equation*}
Fubini's theorem implies that
\begin{align*}
(\mathcal M \boldsymbol{H}h)(s)
&=\frac{c}{\sqrt{2\pi}}\int_{\mathbb R_{+}}x^{-1/2-is}\left(\int_{\mathbb R_{+}}(x+y)^{-1+ia}h(y)dy\right)dx\\
&=\frac{c}{\sqrt{2\pi}}\int_{\mathbb R_{+}}h(y)\left(\int_{\mathbb R_{+}}(x+y)^{-1+ia}x^{-1/2-is}dx\right)dy\\
&\overset{x=ty}{=}c\frac{1}{\sqrt{2\pi}}\int_{\mathbb R_{+}}y^{-1/2-i(s-a)}h(y)dy\int_{\mathbb R_{+}}t^{-1/2-is}(1+t)^{-1+ia}dt\\
&=\frac{c\Gamma(1/2-is)\Gamma(1/2+i(s-a))}{\Gamma(1-ia)}(\mathcal M h)(s-a).
\end{align*}
By the definition of $\boldsymbol{H}
= \mathcal M^{-1}W_{\phi,a}\mathcal M,$
it follows that \[\phi(s)=\frac{c\Gamma(1/2-is)\Gamma(1/2+i(s-a))}{\Gamma(1-ia)}.\]
Using Euler's reflection formula \[\Gamma(z)\Gamma(1-z)=\frac{\pi}{\sin(\pi z)}, \qquad z\notin \mathbb Z,\]
we have
\[|\Gamma(1/2-is)|^2=\frac{\pi}{\sin(\pi(\frac{1}{2}-is))}=\frac{\pi}{\sin\frac{\pi}{2}\cos(ib\pi)}=\frac{\pi}{\cosh (\pi s)}.\]
Similarly, $|\Gamma(1/2+i(s-a))|^2=\frac{\pi}{\cosh \pi (s-a)}.$
Hence,
\[
|\phi(s)|=
\frac{|c|\pi}{|\Gamma(1-ia)|\sqrt{\cosh(\pi s)\cosh(\pi(s-a))}}\in L^1(\mathbb R),
\]
and $\phi\in C_{0}(\mathbb R).$ By \eqref{norm}, 
the operator $\boldsymbol{H}$ is bounded and $\|\boldsymbol{H}\|=\|\phi\|_{\infty}.$
\end{proof}
Define 
\[(\boldsymbol{H}_{a}f)(x)=\int_{\mathbb R_{+}}(x+y)^{-1+ia}f(y)dy.\]
\begin{remark}
The kernel of $\boldsymbol{H}_{a}$ is $k(t)=t^{-1+ia}.$
\begin{itemize}
  \item When $a=0,$ $\boldsymbol{H}_{0}$ is the Carleman operator. It is not compact and its spectrum is $[0,\pi]$ (see e.g. \cite[Section 10.2]{Peller2003}).
  \item Pushnitski and Sobolev \cite{PushnitskiSobolev2025} study a class of self-adjoint Hankel operators involving the 
  kernel function $h(t)=\frac{e^{i\omega\ell\log t}}{t}.$ $k$ and $h$ are functions of the same type.
\end{itemize}
\end{remark}

Recall that the Laguerre functions
\begin{align}
\ell_n(x)=\frac{e^{x/2}}{n!}\frac{d^n}{dx^n}\bigl(e^{-x}x^n\bigr)
\end{align}
form an orthonormal basis of $L^2(\mathbb R_{+})$ (see \cite[(5.1.1) and p108]{Szego1975}).
The following formula is useful in the proof of Theorem \ref{main2}.
\begin{lemma}\cite[p.217]{Rainville1960}
For all nonnegative integers $n,m$,
\begin{align}
(\ell_n\ast \ell_m)(t)=\ell_{n+m}(t)-\ell_{n+m+1}(t).
\end{align}
\end{lemma}
\begin{theorem}\label{main2}
Let $a$ be a nonzero real number and define $\psi(z)=\left(\frac{1-z}{1+z}\right)^{ia}.$
Then $-2^{ia}\Gamma(ia)H_{\psi}$ is  unitarily equivalent  to $\boldsymbol{H}_a.$
In this case, $H_{\psi}$ is a noncompact quasinilpotent Hankel operator.
\end{theorem}
\begin{proof}
The operator $\boldsymbol{H}_a$ has the following matrix representation with respect to the orthogonal basis $\{\ell_n\}$ of $L^2(\mathbb R_{+}).$
For $n,m\geqslant0,$ we have
\begin{equation}
  \begin{aligned}
  \langle \boldsymbol{H}_a\ell_n,\ell_m\rangle
&=\int_{\mathbb R_{+}} \int_{\mathbb R_{+}}(x+y)^{-1+ia}\ell_n(x)\ell_m(y)dxdy\\
&=\int_{\mathbb R_{+}}t^{-1+ia}(\ell_n\ast \ell_m)(t)dt\\
&=\int_{\mathbb R_{+}}t^{-1+ia}\left(\ell_{n+m}(t)-\ell_{n+m+1}(t)\right)dt.
  \end{aligned}
\end{equation}
Let $\alpha_n=\int_{\mathbb R_{+}}t^{-1+ia}\left(\ell_{n}(t)-\ell_{n+1}(t)\right)dt.$
Recall the generating function for $\ell_n$ (see \cite[(5.1.9)]{Szego1975}):
\begin{align}
\sum_{n=0}^{\infty}\ell_n(t)z^{n}=\frac{e^{-t/2}}{1-z}\exp\left(-\frac{tz}{1-z}\right).
\end{align}
A direct calculation gives
\begin{equation}
\begin{aligned}
  \sum_{n=0}^{\infty}\left(\ell_{n}(t)-\ell_{n+1}(t)\right)z^n
&=\sum_{n=0}^{\infty}\ell_{n}(t)z^n-\frac{1}{z}(\sum_{n=0}^{\infty}\ell_{n}(t)z^n-\ell_0(t))\\
&=(1-\frac{1}{z})\sum_{n=0}^{\infty}\ell_{n}(t)z^n+\frac{1}{z}e^{-t/2}\\
&=\frac{z-1}{z}\cdot\frac{e^{-t/2}}{1-z}\exp\left(-\frac{tz}{1-z}\right)+\frac{1}{z}e^{-t/2}\\
&=\frac{1}{z}\left(e^{-t/2}-\exp\left(-\frac{(z+1)t}{2(1-z)}\right)\right).
\end{aligned}
\end{equation}
Let $q(z)=\frac{1+z}{2(1-z)}.$ We have $\operatorname{Re}q(z)=\frac{1-|z|^2}{2|1-z|^2}>0.$
\begin{equation}
  \begin{aligned}
    \psi_1(z)\triangleq z\sum_{n=0}^{\infty}\alpha_nz^n
    &=\int_{\mathbb{R_{+}}}t^{-1+ia}(e^{-t/2}-e^{-q(z)t})dt\\
    &=\frac{1}{ia}\int_{\mathbb{R_{+}}}t^{ia}\left(\frac{1}{2}e^{-t/2}-q(z)e^{-q(z)t}\right)dt\\
    &=\frac{1}{ia}\left(2^{ia}\Gamma(ia+1)-q(z)^{-ia}\Gamma(ia+1)\right)\\
    &=\frac{\Gamma(ia+1)}{ia}(2^{ia}-q(z)^{-ia})\\
    &=\Gamma(ia)\left(2^{ia}-\left(\frac{1+z}{2(1-z)}\right)^{-ia}\right)\\
    &=2^{ia}\Gamma(ia)\left(1-\left(\frac{1+z}{1-z}\right)^{-ia}\right).
  \end{aligned}
\end{equation}
Hence, $H_{\psi_1}=-2^{ia}\Gamma(ia)H_{\psi}.$
\end{proof}

\noindent\textbf{AI Disclosure.}
The weighted composition operator in Proposition \ref{w} was identified with the assistance of AI tools. 
After formulating the problem and interacting with AI tools over several rounds, 
the authors found that the weighted translation operators are noncompact and quasinilpotent. 
The authors subsequently verified the proofs and confirmed that these operators are unitarily 
equivalent to a class of integral Hankel operators via the Mellin transform.

\end{document}